\documentclass[onecolumn]{IEEEtran}
\usepackage{amsthm}
\usepackage{amsmath}
\usepackage{amssymb}
\usepackage{graphicx}
\usepackage[unicode=true]{hyperref}
\usepackage{epstopdf}
\usepackage{algorithm}
\usepackage{algorithmic,color}

\renewcommand{\baselinestretch}{1.65}

\newtheorem{theorem}{Theorem}

\newcommand{\bff}{{\bf f}}

\newcommand{\bv}{{\bf v}}
\newcommand{\fot}{{\frac{1}{2}}}

\begin{document}

\title{Vessel Segmentation in Medical Imaging Using a Tight-Frame Based Algorithm}
\author{Xiaohao Cai, Raymond Chan,\thanks{Department of Mathematics, The Chinese University of Hong Kong,
Shatin, Hong Kong, China (xhcai@math.cuhk.edu.hk and rchan@math.cuhk.edu.hk).}
Serena Morigi, Fiorella Sgallari\thanks{Department of Mathematics-CIRAM, University of Bologna, Bologna, Italy  (morigi@dm.unibo.it and sgallari@dm.unibo.it).}
 }
\date{}
\maketitle

\begin{abstract}
Tight-frame, a generalization of orthogonal wavelets, has been used
successfully in various problems in image processing, including
inpainting, impulse noise removal, super-resolution image restoration,
etc. Segmentation is the process of identifying object outlines
within images. There are quite a few efficient algorithms for segmentation
that depend on the variational approach and
the partial differential equation (PDE) modeling.
 In this paper, we propose to apply the
tight-frame approach to automatically identify tube-like structures
such as blood vessels in Magnetic Resonance Angiography (MRA) images.
Our method iteratively refines a region that encloses the possible
boundary or surface of the vessels. In each iteration, we apply
the tight-frame algorithm to denoise and smooth the possible
boundary and sharpen the region. We prove the convergence of our
algorithm. Numerical experiments on real 2D/3D MRA images demonstrate
that our method is very efficient with convergence usually
within a few iterations, and it outperforms existing PDE and
variational methods as it can extract more tubular objects
and fine details in the images.
\end{abstract}

\vskip1cm



\section{Introduction}\label{sec1}

Segmentation problem of branching tubular
objects in 2D and 3D images  arises in many
applications, for examples, extracting roads in aerial photography,
and anatomical surfaces of blood vessels in
medical images. In this paper, we are concerned with identifying 
tube-like structures in Magnetic Resonance Angiography (MRA) 
images. Because of the necessity to
obtain as much fine details as possible in real time, automatic, robust
and efficient methods are needed. However, due to the low contrast,
intensity inhomogeneity, partial occlusions and intersections, and the
presence of noise and possible blur in the images, this segmentation
problem is very challenging.

There are many different approaches for vessel segmentation, see
for example \cite{CV,DCS,ESF01,GH,KM,LF,SC} and the extended reviews
\cite{CRR,CF}. In the following, we concentrate on two approaches
that are related to our approach:
the active contour approach and the partial
differential equation (PDE) approach.
For the vessel segmentation algorithms
that are based on deformable models \cite{MT}, because
the explicit deformable model representation is usually impractical,
level set techniques to evolve a deformable model have been introduced,
and they provide implicit representation of the deformable model.
However, the level set segmentation
approach is computationally more expensive as it needs to cover
the entire domain of interest, which is generally one dimension
higher than the original one. Interested readers can refer
to recent literature on the level set segmentation strategy for
tubular structures \cite{GH,HF,S}.

Based on the curve evolution techniques,
Mumford-Shah functional and level sets,
a new model for active contours to detect objects
in a given image was proposed in \cite{CV}.
Unlike the classical active contour models, this model does
 not depend on the gradient of the image.
 In \cite{BC}, a generalization of the active contour
without edges model was proposed for object detection using logic operations.
However, both the active contour model \cite{CV} and the logic framework
\cite{BC} are not suitable for detecting tubular structures with low
contrast and intensity inhomogeneity, see numerical results in Section \ref{sec5}.

In \cite{ESF01}, a geometric deformable model
for the segmentation of tubular-like structures was
proposed. The model is characterized mainly by two
components by using a suitable diffusion tensor:
the mean curvature flow and the
directionality of the tubular structures. The major
advantage of this technique is the ability to segment
twisted, convoluted and occluded structures without
user interactions; and it can follow the branching of different
layers, from thinner to larger structures. The
dependence on the grid resolution chosen to solve
the discretized PDE model is still an open problem.
The authors in \cite{ESF01} have also applied a variant
of the proposed PDE model to the challenging problem
of composed segmentation in \cite{ESF02}.

Besides the methods above,
there were some initial work by using wavelets or tight-frames
to do texture classification and segmentation \cite{AG,U}.
The tight-frame
approach is a versatile and effective tool for many different
applications in image processing, see
\cite{CSS03,JRZ,RSG,CCSS,SS,SCD,YL}.
There are many kinds of tight-frame systems, such as
those from framelets \cite{RS},  contourlets \cite{DV} and curvelets
\cite{CD,CDDL}, etc.
Recently, the authors in \cite{DCS} proposed to combine
the tight-frame image restoration model of \cite{COS}
and the total variation based segmentation model
of \cite{BEVTO,CEM,CV} to do segmentation. Their
approach results in a minimization problem. In this paper,
we also derive a segmentation algorithm that uses the
tight-frames. However our method is not based
on minimizing any variational model and hence it is
different from the method in \cite{DCS}. In fact, our algorithm
iteratively updates the set of possible boundary pixels
to change the given image into a binary image. Like the method in \cite{ESF01},
our method also has the ability to segment twisted, convoluted
and occluded structures. We will show that our algorithm is convergent.
In fact, numerical results show that it converges within 10
iterations for 2D as well as 3D Magnetic Resonance Angiography (MRA) images.
We will see that our method can
extract much more details from the given image than the
method in  \cite{CV}, \cite{DCS} and \cite{ESF01}.
We remark that a preliminary version of our segmentation algorithm
has been given in our proceeding paper \cite{CCMS}. The main 
contributions in this paper are: (1) a simpler strategy to initialize
and refine the regions enclosing the possible boundary 
of the vessels, (2) the new strategy leads to a simple proof of 
convergence and easier choice of parameters, (3) a different tight-frame
with better directionally selective property is used to 
obtain more details in the image, (4) a new 3D result with higher noise
and more complicated vessel is added.

The rest of the paper is organized as follows. In Section \ref{sec2},
we recall some basic facts about tight-frame and
tight-frame algorithms. Our segmentation algorithm is given in
Section \ref{sec3}. In Section \ref{sec5} we test our algorithm 
on various real 2D and 3D MRA images and compare it 
with those representative algorithms from different approaches:
PDE-based \cite{ESF01}, tight-frame \cite{DCS} and active contour
models \cite{CV}. Conclusions are given in Section \ref{sec6}.

\section{Tight-Frame Algorithm} \label{sec2}

In this section, we briefly introduce the tight-frame algorithm
which are based on tight-frame transforms.
All tight-frame transforms ${\cal A}$ have a very important property,
the ``{\it perfect reconstruction property}'':
${\cal A}^T{\cal A} = {\cal I}$, the identity transform, see \cite{RS}.
Unlike the wavelets, in general, ${\cal A}{\cal A}^T \not = {\cal I}$.
For theories of framelets and tight-frame transforms, we refer the
readers to \cite{DI} for more details. In order to apply the tight-frame
algorithm, one only needs to know the filters corresponding to the framelets
in the tight-frame. In the followings, we give two examples of
tight-frames: the piecewise linear B-spline tight-frame \cite{DHRS} and
the dual-tree complex wavelet tight-frame \cite{SBK}.

The filters in the piecewise linear B-spline tight-frame are:
\begin{equation} \label{filter}
h_0 = \frac{1}{4} [1,\  2,\ 1], \quad h_1 = \frac{\sqrt{2}}{4} [1,\  0,\ -1],
\quad  h_2 = \frac{1}{4} [-1,\  2,\ -1],
\end{equation}
see \cite{RS}. The tight-frame coefficients
of any given vector $\bv$ corresponding to filter $h_i$ can be obtained
by convolving $h_i$ with $\bv$. In matrix terms, we can
construct, for each filter, its corresponding
filter matrix which is just the Toeplitz matrix with diagonals
given by the filter coefficients, e.g. $H_0=\frac{1}{4}{\rm tridiag}[1, \ 2, \ 1]$.
Then the 1D tight-frame forward transform  is given by
\begin{equation} \label{transform}
{\cal A} =
\begin{bmatrix}
H_0 \\
H_1 \\
H_2
\end{bmatrix}
.
\end{equation}
To apply the tight-frame transform onto $\bv$ is equivalent to computing
${\cal A}\bv$, and $H_i \bv$ gives the tight-frame coefficients corresponding
to the filter $h_i$, $i=1,2,3$.

The $d$-dimensional  piecewise linear B-spline tight-frame
is constructed by tensor products from the 1D tight-frame
above, see \cite{DS}.
For example, in 2D, there are nine  filters
given by $h_{ij}\equiv h_i^T \otimes h_j$ for $i, j = 1, 2, 3$, where
$h_i$ is given in (\ref{filter}). For any 2D image $f$, the tight-frame
coefficients with respect to $h_{ij}$ are obtained by
convolving $h_{ij}$ with $f$. The corresponding forward transform
${\cal A}$ will be a stack of nine block-Toeplitz-Toeplitz-block matrices
(cf. (\ref{transform})). The tight-frame coefficients are given by the
matrix-vector product ${\cal A} \bff $,
where  $\bff = \textrm{vec}(f)$
denotes the vector
obtained by concatenating the columns of $f$.

Dual-tree complex wavelet transform (D$\mathbb{C}$WT)
was firstly introduced by Kingsbury \cite{K1,K2}.
Apart from having the usual perfect reconstruction property, shift-invariance
property and linear complexity, it also has the nice directionally selective
property  at $\pm 15^\circ, \pm 45^\circ, \pm 75^\circ$.
The idea is to use two different sets of filters: one gives the real
part of the transform and the other gives the imaginary part.
 Let \{$g_0, g_1$\} and \{${h}_0, {h}_1$\}
denote the two different sets of  orthonormal filters,
where $g_0, {h}_0$
are the low pass filters and  $g_1, {h}_1$ are the high
pass filters. Let the square matrix
$A_{g_i {h}_j}$ denote the 2D separable wavelet transform implemented
using $g_i$ along the rows and ${h}_j$ along the columns,
and define $A_{{h}_i{h}_j}, A_{g_ig_j}, A_{{h}_ig_j}$
similarly, where $i, j = 1, 2$. Then the forward
transform for the 2-dimensional
D$\mathbb{C}$WT is represented by
\[
{\cal A} \bff :=
\frac{1}{\sqrt{8}}
\begin{bmatrix}
I & -I  & 0 & 0\\
I & I & 0 & 0 \\
0 & 0& I & I  \\
0 & 0 & I & -I
\end{bmatrix}
\begin{bmatrix}
A_{g_ig_j}  \\
A_{ {h}_i {h}_j}  \\
A_{ {h}_ig_j}  \\
A_{g_i {h}_j} \\
\end{bmatrix}
\bff,
\quad
i, j = 1, 2.
\]
We see that the 2-dimensional D$\mathbb{C}$WT requires four different wavelet
transforms in parallel. We refer the readers to \cite{CWTweb,DI, GNG,SBK}
and the references therein for more details and for the implementation of
3-dimensional D$\mathbb{C}$WT. In practice, the D$\mathbb{C}$WT are implemented by
using different sets of filters for the first level and the remaining levels,
see \cite{SBK}. The filters and the Matlab code for D$\mathbb{C}$WT can
be obtained from \cite{CWTweb}.

The tight-frame algorithms, as given in \cite{CSS03,CCSS08,CCSS09,JRZ,CCSS},
are of the following generic form:
\begin{eqnarray}
{\bff}^{(i+\fot)} &=& \mathcal{U}({\bff}^{(i)}), \label{q1}  \\
{\bff}^{(i+1)} &=& {\cal A}^T\mathcal{T}_\lambda({\cal A} {\bff}^{(i+\fot)}),\quad
i=1, 2,\ldots. \label{q2}
\end{eqnarray}
Here $\bff^{(i)}$ is an approximate solution at the $i$-th iteration,
$\mathcal{U}$ is a problem-dependent operator, and
$\mathcal{T}_\lambda(\cdot)$ is
the soft-thresholding operator defined as follows.
Given vectors $\bv=[v_1, \cdots, v_n]^T$ and
${\bf \lambda}=[\lambda_1, \cdots, \lambda_n]^T$,
$\mathcal{T}_\lambda(\bv) \equiv
[t_{\lambda_1}(v_1), \cdots, t_{\lambda_n}(v_n)]^T$,
where
\begin{equation}\label{ltk}
t_{\lambda_k}(v_k) \equiv
\left\{
\begin{array}{lcl}
\textrm{sgn}(v_k)(|v_k|-\lambda_k), & &  \text{if}\  |v_k| > \lambda_k, \\
0, \qquad \qquad \qquad  & &  \text{if}\  |v_k| \leq \lambda_k.
\end{array}
\right.
\end{equation}
For how to choose $\lambda_k$, see \cite{D}.

We remark that (\ref{q2}) performs a tight-frame denoising and smoothing
on the image while (\ref{q1})
performs a data-fitting according to the specific problem at hand.
For high-resolution image reconstruction problem \cite{CCSS},
astronomical infra-red imaging \cite{CCSS08},
impulse noise removal \cite{CCSS09}, and inpainting
problem \cite{CSS03}, the tight-frame algorithm (\ref{q1})--(\ref{q2})
has been shown to be convergent to a functional with the
regularization term being the
1-norm of the tight-frame coefficients, see \cite{JRZ}.

\section{Tight-Frame Based Algorithm for Segmentation}\label{sec3}

The technology of Magnetic Resonance Angiography (MRA) imaging is based
on detection of signals from flowing blood and suppression of signals
from other static tissues, so that the blood vessels appear as
high intensity regions in the
image, see Fig. \ref{fig_3parts}(a). The structures to be segmented are
vessels of variable diameters which are close to each
other.  In general in medical images, speckle noise and weak edges make it
difficult to identify the structures in the  image. Fortunately, the MRA
images contain some properties that can be exploited
to derive a good segmentation algorithm. From Fig. \ref{fig_3parts}(a),
we see that the pixels near the boundary of the vessels
are not exactly of one value, but they are in some range, whereas
the values of the pixels in other parts are far from this range. Thus
the main idea of our algorithm is to approximate this range accurately.
We will obtain the range iteratively by a tight-frame algorithm.
The main steps are as follows. Suppose in the beginning
of the $i$th iteration, we are given an approximate image $f^{(i)}$
and a set $\Lambda^{(i)}$ that contains all possible boundary pixels.
Then we (i) use $\Lambda^{(i)}$ to estimate an appropriate range
$[\alpha_i, \beta_i]$  that contains the pixel values of possible
boundary pixels; (ii) use the range to separate the image into
three parts---those below the range (background pixels), inside
the range (possible boundary pixels), and above the
range (pixels in the vessels);
(iii) denoise and smooth the inside part by the tight-frame
algorithm to get a new image $f^{(i+1)}$.
We stop when the image becomes binary and this
happens within 10 iterations for the real MRA images
we tested, see Table \ref{cardi}
in Section \ref{sec5}.
In the followings, we elaborate each of the steps.
Without loss of generality, we assume that the given image $f$
has dynamic range in [0, 1].

\vspace{3mm}
\noindent{\bf Initialization.} To start the algorithm at $i=0$, we need
to define $f^{(0)}$, the initial guess, and $\Lambda^{(0)}$,
the initial set of possible boundary pixels. Naturally, we set $f^{(0)} =f$,
the given image.
For $\Lambda^{(0)}$, since we do not have any knowledge of where the
boundary pixels are at the beginning, we identify them by using the gradient
of $f$, i.e. we locate them as pixels where the gradient is bigger than a
threshold $\epsilon$. More precisely, let $\Omega$ be the index set of all the
pixels in the image, then we define
\begin{equation} \label{lambda0}
\Lambda^{(0)} \equiv \{j \in \Omega  \ | \ \| [\nabla f]_j\|_1 \ge \epsilon\}.
\end{equation}
Here $[\nabla f]_j$ is the discrete gradient of $f$ at the $j$th pixel.
Once $f^{(0)}$ and $\Lambda^{(0)}$ are defined, we can start the iteration.

\vspace{3mm}
\noindent{\bf Step (i): computing the range $[\alpha_{i}$, $\beta_{i}]$.}
Given $\Lambda^{(i)}$, we first compute the mean pixel value on $\Lambda^{(i)}$:
\begin{equation} \label{mui}
\mu^{(i)} = \frac{1}{|\Lambda^{(i)}|} \sum_{j\in \Lambda^{(i)}} f^{(i)}_j,
\end{equation}
where $| \cdot |$ denotes the cardinality of the set
and $f^{(i)}_j$ is the pixel value of pixel $j$ in image $f^{(i)}$.
Then we compute the mean pixel values of the two sets separated by $\mu^{(i)}$:
\begin{equation} \label{muminus}
\mu^{(i)}_{-} = \frac{1}{|\{ j\in \Lambda^{(i)}: f ^{(i)}_j \leq \mu^{(i)}\}|}
\sum_{\{ j\in \Lambda^{(i)}: f ^{(i)}_j \leq \mu^{(i)}\}}f ^{(i)}_j,
\end{equation}
and
\begin{equation} \label{muplus}
\mu^{(i)}_{+} =
\frac{1}{|\{ j\in \Lambda^{(i)}: f ^{(i)}_j \geq \mu^{(i)}\}|}
\sum_{\{ j\in \Lambda^{(i)}: f ^{(i)}_j \geq \mu^{(i)}\}}f ^{(i)}_j.
\end{equation}
While $\mu^{(i)}$ reflects the mean energy
of the set of possible boundary pixels, $\mu^{(i)}_{-}$ and $\mu^{(i)}_{+}$ reflect
the mean energies of the pixels on the boundary closer to the background
and closer to the vessels respectively. We define
\begin{equation}
\alpha_i \equiv \max\left\{\frac{\mu^{(i)} + \mu^{(i)}_{-}}{2},0\right\}, \quad
\beta_i \equiv  \min\left\{\frac{\mu^{(i)}+ \mu^{(i)}_{+}}{2},1\right\}.
\label{int_i}
\end{equation}

\vspace{3mm}
\noindent
{\bf Step (ii): thresholding the image into three parts.}
Using the range $[\alpha_i, \beta_i] \subseteq [0, 1]$, we can separate
the image $f^{(i)}$ into three parts---those below, inside, and above the
range, see Fig. \ref{fig_3parts}(b) for $i=0$. Since our aim is to create a
binary image, we threshold those pixel
values that are smaller than $\alpha_i$ to 0, those larger than $\beta_i$ to 1,
and those in between, we stretch them between 0 and 1 using
a simple linear contrast stretch, see \cite{GW}. If there
are no pixels in between $\alpha_i$ and $\beta_i$, then the
threshold image is binary and the algorithm stops.
More precisely, let
\begin{eqnarray}
M_i = {\rm max}\{ f^{(i)}_j \ | \      \alpha_i \le f^{(i)}_j \le
\beta_i, j \in {\Lambda^{(i)}} \}, \quad
m_i ={\rm min}\{ f^{(i)}_j \ | \      \alpha_i \le f^{(i)}_j \le
\beta_i, j \in {\Lambda^{(i)}} \}, \label{Mm}
\end{eqnarray}
then we define
\begin{equation}
f^{(i+\fot)}_j =
\left\{
\begin{array}{lclll}
0, & & {\rm if} \  f^{(i)}_j  \le \alpha_i, \\
\frac{f^{(i)}_j - m_i}{M_i-m_i}, & &
 \alpha_i \le f^{(i)}_j \le \beta_i, & & \quad {\rm for\ all} \ j \in {\Omega}.
\label{eq5} \\
1, & & {\rm if} \  \beta_i \le  f^{(i)}_j ,
\end{array}
\right.
\end{equation}

Fig. \ref{fig_3parts}(c) shows the threshold and stretched
image from Fig. \ref{fig_3parts}(b), where the yellow pixels
are pixels we have classified not on the boundary,
i.e. $f^{(i+\fot)}_j=0$ (signifying pixel $j$ is
in the background) or $f^{(i+\fot)}_j=1$
(signifying pixel $j$ is inside the vessel).
The remaining pixels are remained to be classified and we
denote the set by:
\begin{equation}\label{Lambdai}
\Lambda^{(i+1)}=\{ j \ | \ 0 < f^{(i+\fot )}_j < 1 , j \in {\Omega} \}.
\end{equation}
Note that those
pixels with values $m_i$ and $M_i$ are mapped to 0 and 1 respectively
and hence will not be in $\Lambda^{(i+1)}$.
Next we denoise and smooth
$f^{(i+\fot)}$ on $\Lambda^{(i+1)}$.

\begin{figure}
\begin{center}
\begin{tabular}{ccc}
\includegraphics[scale=0.5]{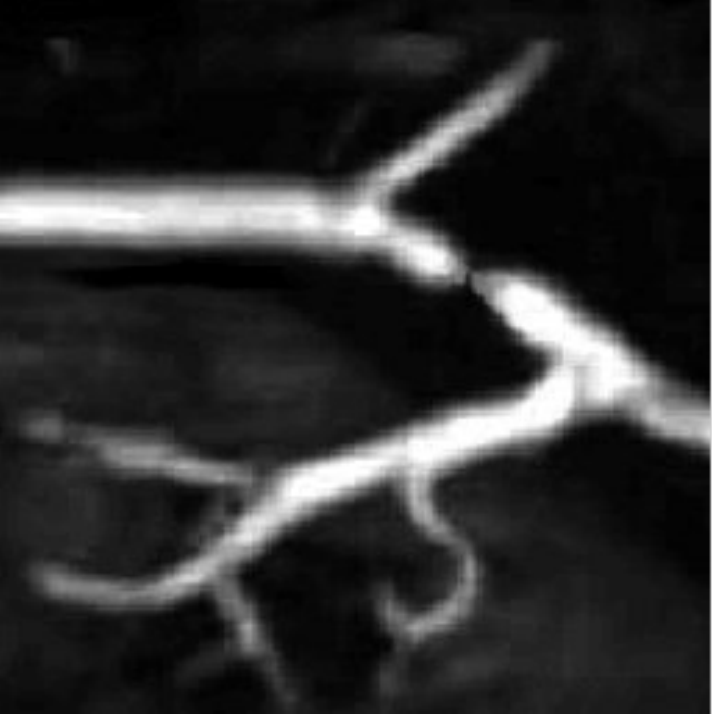}&
\includegraphics[scale=0.5]{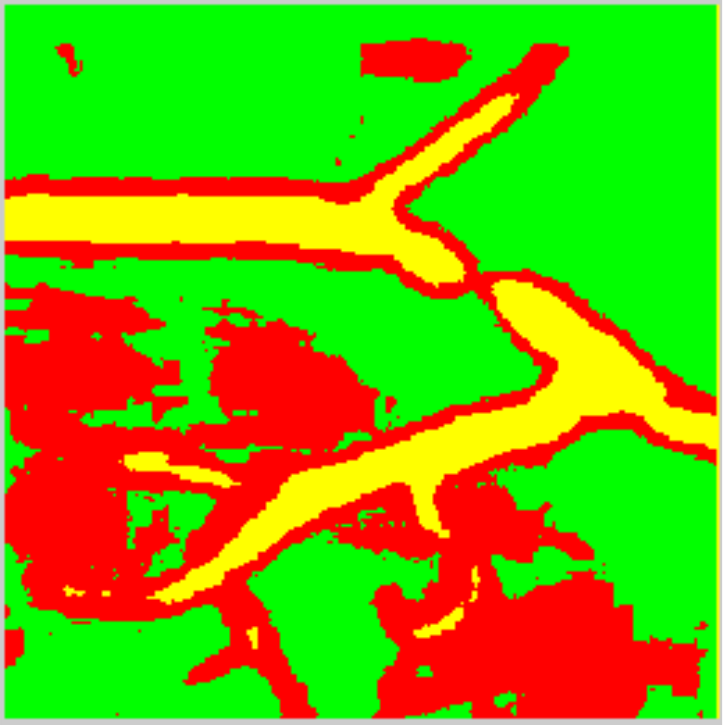}&
\includegraphics[scale=0.5]{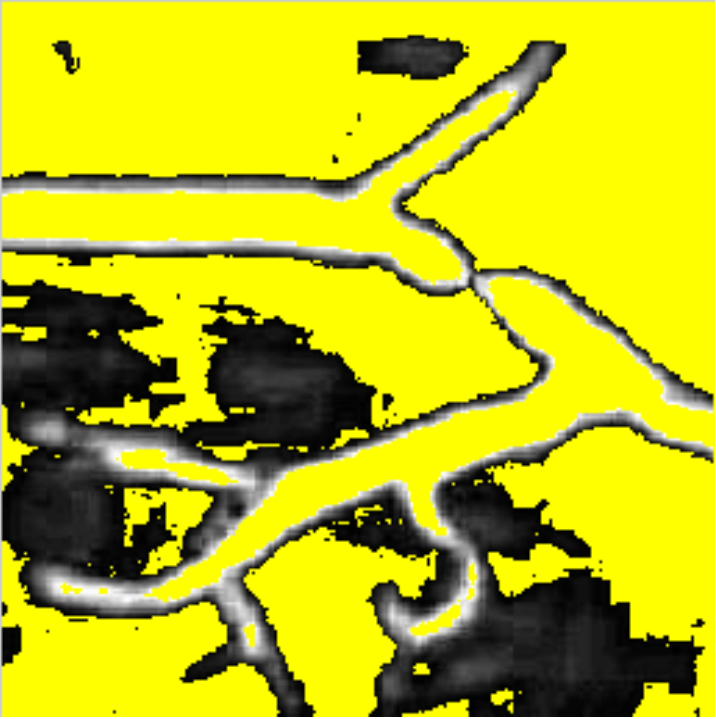}\\
(a)&(b)&(c)
\end{tabular}
\end{center}
\caption{(a) Given MRA image. (b) Three parts of the given image
(green--below, red--in between, and yellow--above).
(c) Threshold and stretched image by (\ref{eq5})
(yellow pixels are with value 0 or 1).} \label{fig_3parts}
\end{figure}

\vspace{3mm}\noindent
{\bf Step (iii): tight-frame iteration.}
To denoise and smooth the image $f^{(i+\fot)}$ on $\Lambda^{(i+1)}$,
we apply the tight-frame iteration (\ref{q2}) on $\Lambda^{(i+1)}$.
More precisely, if $j \not \in \Lambda^{(i+1)}$, then
we set $f^{(i+1)}_j=f^{(i+\fot)}_j$; otherwise, we use (\ref{q2}) to get
$f^{(i+1)}_j$. To write it out clearly,
let $\bff^{(i+\fot)}={\rm vec}(f^{(i+\fot)})$,
and $P^{(i+1)}$ be the diagonal matrix where the
diagonal entry is 1 if the corresponding index is in $\Lambda^{(i+1)}$,
and 0 otherwise. Then
\begin{equation} \label{tfi}
\bff^{(i+1)} \equiv (I-P^{(i+1)})\bff^{(i+\fot)} + P^{(i+1)}
{\cal A}^T\mathcal{T}_{\lambda}({\cal A}\bff^{(i+\fot)}).
\end{equation}
By reordering the entries of the vector $\bff^{(i+1)}$
into columns, we obtain the image $f^{(i+1)}$.
Note that the effect of (\ref{tfi}) is to denoise and smooth
the image on $\Lambda^{(i+1)}$, see \cite{JRZ}.
Since the pixel values of all pixels outside $\Lambda^{(i+1)}$
are either $0$ or $1$, the cost of the tight-frame transform
in (\ref{tfi}), such as the computation of ${\cal A}\bff^{(i+\fot)}$,
can be reduced significantly by taking advantage of the fact that all computations
can be done only on pixels around $\Lambda^{(i+1)}$.

\vspace{3mm}
\noindent{\bf Stopping criterion.}
We stop the iteration when all the pixels of $f^{(i+\fot)}$ are
either of value $0$ or $1$, or equivalently when
$\Lambda^{(i)}=\emptyset$. For the binary image $f^{(i+\fot)}$,
all the pixels with value 0 are considered as background
pixels and pixels with value 1 constitute the tubular structures.

Below we give the full algorithm and show that
it always converges to a binary image.

\begin{center}
{\bf Algorithm 1:} Tight-frame algorithm for segmentation \\
\vspace{0.1in}
\begin{tabular}{rll}\hline
 1.& Input: given image $f$. \cr
 2.& Set $f^{(0)}=f$  and $\Lambda^{(0)}$ by (\ref{lambda0}) \cr
  3.& Do $i=0, 1,\ldots,$ until stop \\
   & (a) Compute $[\alpha_i, \beta_i]$ by $(\ref{int_i})$. \cr
   & (b) Compute $f^{(i+\fot)}$ by (\ref{eq5}).\cr
   & (c) Stop if $f^{(i+\fot)}$ is a binary image. \cr
   & (d) Compute $\Lambda^{(i+1)}$ by (\ref{Lambdai}). \cr
   & (e) Update $f^{(i+\fot)}$ to $f^{(i+1)}$ by (\ref{tfi}).\cr
 4.& Output: binary image $f^{(i+\fot)}$. \cr
 \hline
\end{tabular}
\end{center}
\vspace{0.1in}

\begin{theorem} \label{thm1}
Our tight-frame algorithm will converge to a binary image.
\end{theorem}

\begin{proof} From (\ref{Lambdai}), it suffices to prove that $|\Lambda^{(i)}|=0$
at some finite step $i>0$. By (\ref{Mm}), if $f^{(i+\fot)}$ is not yet
a binary image, then there will be at least one $j \in \Lambda^{(i)}$
such that $f^{(i)}_j =M_i$. By (\ref{eq5}),
 $f^{(i+\fot)} _j$ will be set to 1 and hence by (\ref{Lambdai}),
 $j \not \in
\Lambda^{(i+1)}$.  Hence $|\Lambda^{(i+1)}|< |\Lambda^{(i)}|$.
Since $|\Lambda^{(0)}|$ is finite, there must exist some $i$
such that $|\Lambda^{(i)}|=0$.
\end{proof}

We emphasize that the algorithm actually converges within 10 iterations for
the 2D and 3D real images we have tested.

Finally, let us estimate the computation cost of our method for a given
image with $n$ pixels. Since the costs of computing $\mu^{(i)}$,
$\mu^{(i)}_{-}$, $\mu^{(i)}_{+}$, and $[\alpha_i, \beta_i]$
are all of $O(n)$ operations, see (\ref{mui})--(\ref{int_i}); and the cost of
a tight-frame transform is also linear with respect to $n$
(see e.g. (\ref{transform}) where $H_i$ are all tri-diagonal matrices),
we see that the complexity of our algorithm is $O(n)$
per iteration.
In fact, one can speed up the computation tremendously
as all computations can be done only on pixels around $\Lambda^{(i)}$
 and there is no need to carry out the computations in $\Omega$---though
in the numerical tests, we did not optimize
the code and we just carried out the computations in $\Omega$.
According to Table \ref{cardi}, after just 3 iterations, the set
$\Lambda^{(i)}$ contains
only a hundred pixels in the 2D case and just two thousand pixels in
the 3D case.

\section{Numerical Examples}\label{sec5}

In this section, we try our tight-frame segmentation algorithm
(Algorithm 1) on 2D and 3D real MRA images tested in \cite{ESF01}
and \cite{ESF02}. Our algorithm is written in Matlab.
The thresholding parameters $\lambda_k$ used in (\ref{ltk}) are
all chosen to be $\lambda_k \equiv 0.1$; and
we choose $\epsilon=0.003$ and 0.06 for 2D
and respectively 3D images in (\ref{lambda0}).
To display the boundary and the surface in the binary images visually,
we use the Matlab commands ``{\tt contour}'' and ``{\tt isosurface}''
for 2D and 3D images respectively. The
tight-frame we used is the D$\mathbb{C}$WT downloaded from
\cite{CWTweb} where all parameters in the code are chosen to be the
default values. In particular, the number of wavelet levels
used in the code is 4.
The 2D images are tested in a MacBook with 2.4 GHz processor
and 4GB RAM, while the 3D images, because of their sizes,
are tested on a node with 120GB RAM in a PC-cluster.
We compare our method with those representative
algorithms from different approaches:
PDE-based \cite{ESF01}, tight-frame \cite{DCS} and active contour
models \cite{CV}, where the programs are provided by the authors.

\bigskip
{\it Example 1.}
The test image is a $182 \times 182$ MRA image of a carotid
vascular system, see Fig. {\ref{fig_carotid}}(a). The blood
vessels contain regions with high and low intensities,
including some very thin vessels in the middle with intensities as
low as the intensity of the background. Intersections of partial structures
 even increase the difficulty of the segmentation. The segmentation
 results by different methods are given in Fig. {\ref{fig_carotid}}
where the given image is overlaid so that we can compare
the accuracy of the methods.
Clearly, the results of Figs. {\ref{fig_carotid}}(b) and (c) are not satisfactory
since the vessels obtained in Fig. {\ref{fig_carotid}}(b) are
disconnected and some vessels in Fig. {\ref{fig_carotid}}(c) are not detected.
By comparing the parts inside the rectangles in Fig. {\ref{fig_carotid}}(d)
with those in Fig. {\ref{fig_carotid}}(e), we see that our method can extract
clearer boundary than the method in \cite{ESF01}, especially for
handling the artifacts near the boundary.
We also see that the noise in the image has been removed because
of the denoising property of formula (\ref{q2}) as
explained in Section \ref{sec2}. Our method converges
in 5 iterations and requires 0.64 seconds. The first column of
Table \ref{cardi} gives $|\Lambda^{(i)}|$ at each iteration.

\begin{figure}[!htb]
\begin{center}
\begin{tabular}{ccccc}
\includegraphics[width=53mm, height=53mm]{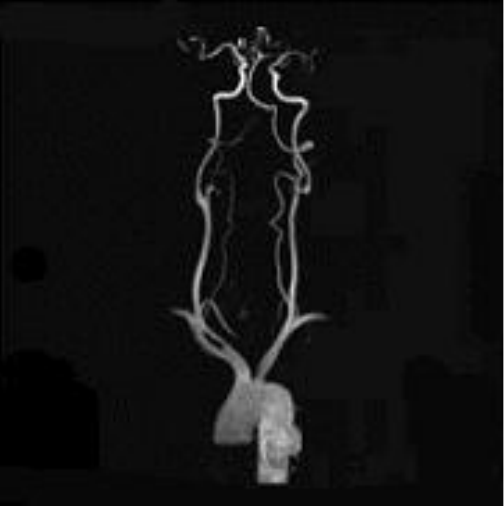}&
\includegraphics[width=26mm, height=53mm]{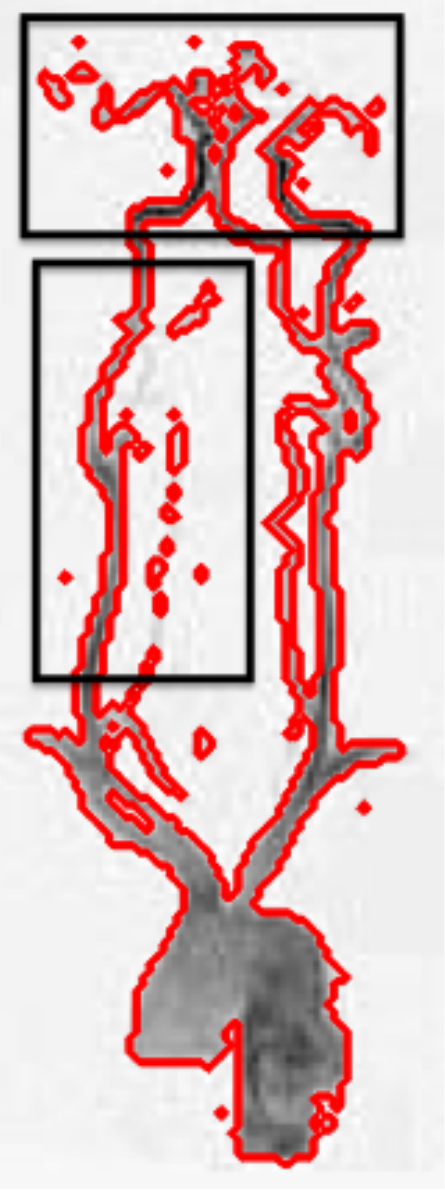}&
\includegraphics[width=26mm, height=53mm]{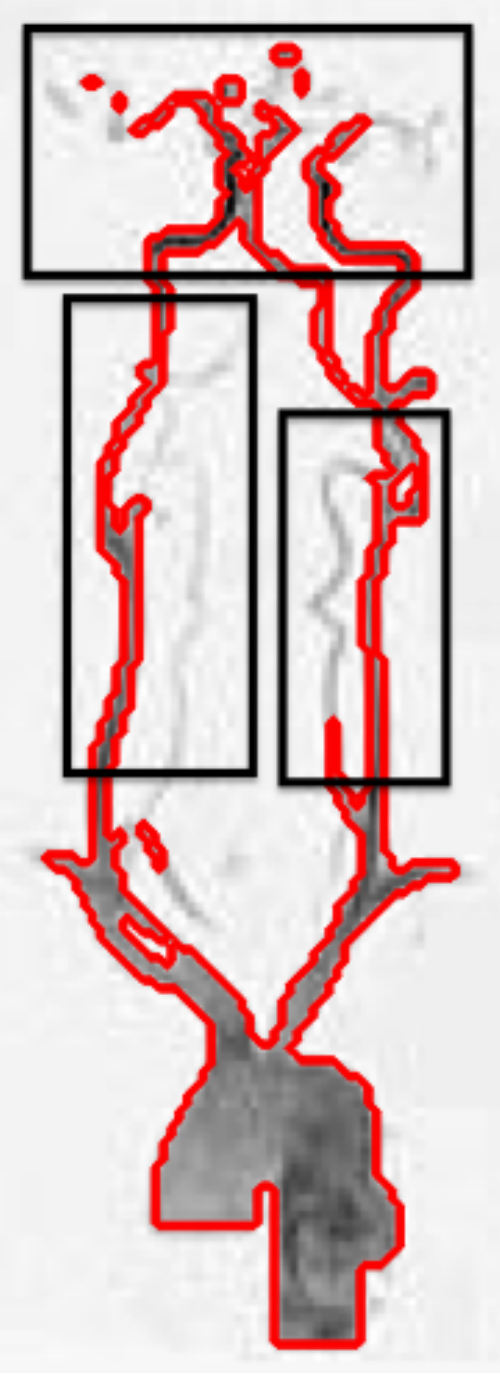}&
\includegraphics[width=26mm, height=53mm]{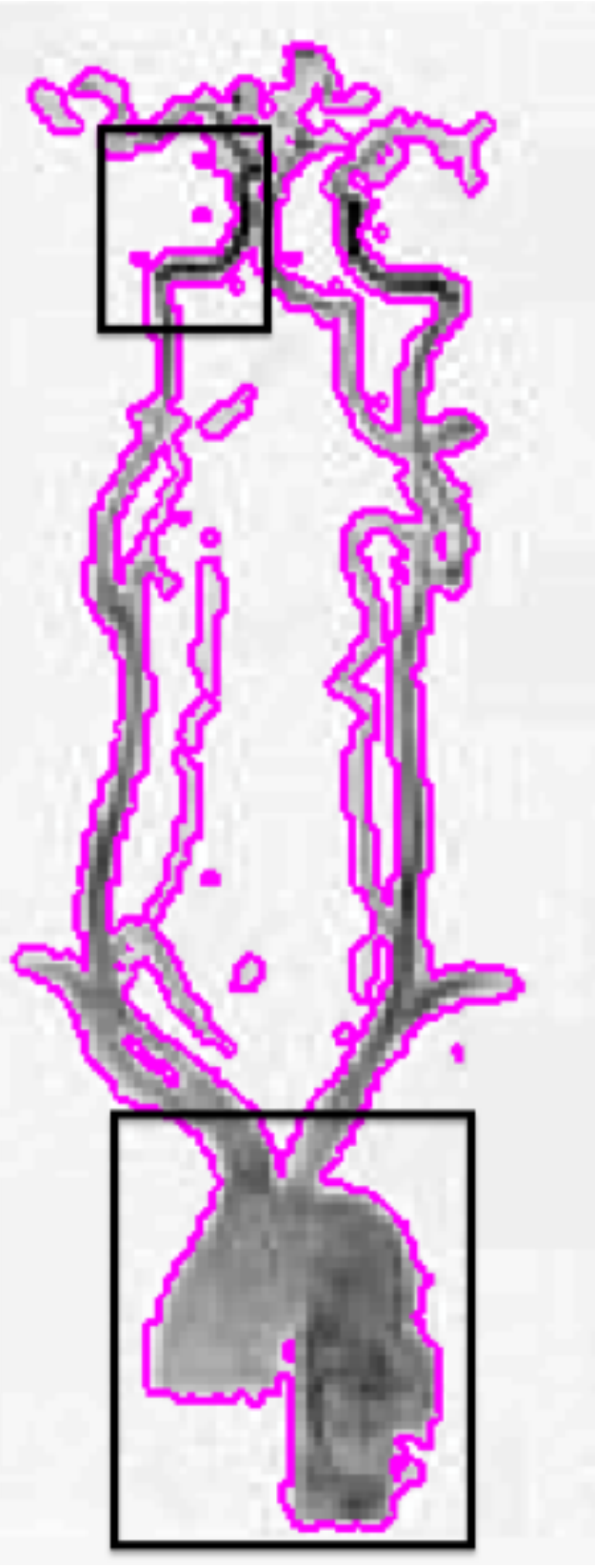}&
\includegraphics[width=26mm, height=53mm]{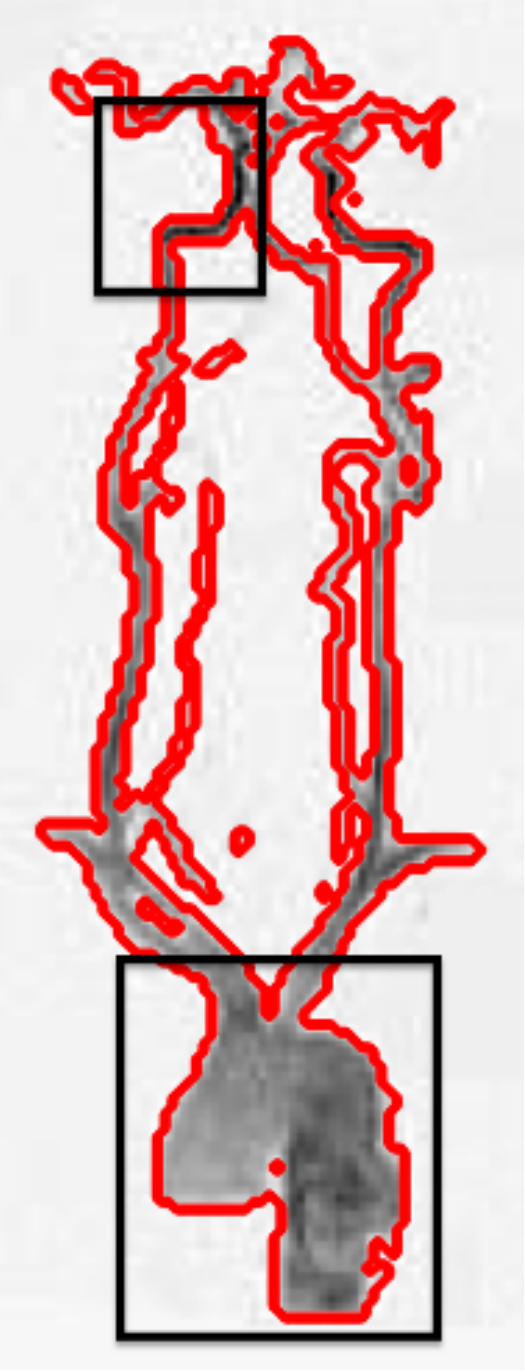}
\\
(a) & (b) & (c) & (d)&(e)
\end{tabular}
\end{center}
\caption{Carotid vascular system segmentation. (a) Given image.
(b), (c) and (d) Results by the methods in \cite{CV}, \cite{DCS} and
\cite{ESF01} respectively. (e) Result of our method.} \label{fig_carotid}
\end{figure}

\begin{table}[!h]
\centering \caption{Cardinality of $\Lambda^{(i)}$ at each iteration of the
four examples}
\vspace{0.1in}
\begin{tabular}{|c||c|c|c|c|}
\hline \vspace{-0.1in}
  & Example 1 & Example 2 & Example 3 & Example 4 \\
$|\Lambda^{(i)}|$ & $|\Omega| = 11284$  & $|\Omega| = 66049$  &
$|\Omega| = 8120601$ & $|\Omega| = 6000000$
\\ \hline
$i = 0$ & 1721  & 9444 & 137330 & 152898
\\ \hline
$i = 1$& 354  & 1943 & 32760 & 32064
\\ \hline
$i = 2$&  82 & 464 & 8795 & 8565
\\ \hline
$i = 3$&  26 & 133 & 2475 & 2391
\\ \hline
$i = 4$&  4 & 32 & 689 & 650
\\ \hline
$i = 5$& 0  & 8 & 189 & 187
\\ \hline
$i = 6$& -  & 0 & 56 & 59
\\ \hline
$i = 7$&  - & - & 15  & 12
\\ \hline
$i = 8$& -  & - & 3 & 2
\\ \hline
$i = 9$& -  & - & 0 & 0
\\ \hline
\end{tabular}
\label{cardi}
\end{table}

\bigskip
{\it Example 2.} The test image is a $256\times 256$ MRA image
of a kidney vascular system as shown in Fig. {\ref{fig_kidney}}(a). This
example shows the ability of our method to reconstruct structures
which present small occlusions along the coherence direction.
Our method converges in 6 iterations with 0.78 seconds;
and the second column in
Table \ref{cardi} gives $|\Lambda^{(i)}|$ at each iteration.
The results by the method in \cite{CV} (Fig. {\ref{fig_kidney}}(b))
and  by the method in \cite{DCS} (Fig. {\ref{fig_kidney}}(c)) are not
good since they can not recover the small occlusions along the
coherence direction, while this can be done by our method and
the method in \cite{ESF01}, see Figs. {\ref{fig_kidney}}(d) and (e).
Furthermore, our method is better than the method in \cite{ESF01}
by comparing the rectangular parts of Fig. {\ref{fig_kidney}}(d)
with those in Fig. {\ref{fig_kidney}}(e), since our method can detect
smoother edges; see Figs. \ref{fig_kidney}(f)--(k)
which are the results of zooming in the
rectangular parts of Figs. {\ref{fig_kidney}}(d) and (e) respectively.
The results also show  that our method is very effective in removing artifacts.

\begin{figure}[!htb]
\begin{center}
\begin{tabular}{ccc}
\includegraphics[width=41mm, height=41mm]{figs/fig1a.pdf}&
\includegraphics[width=41mm, height=41mm]{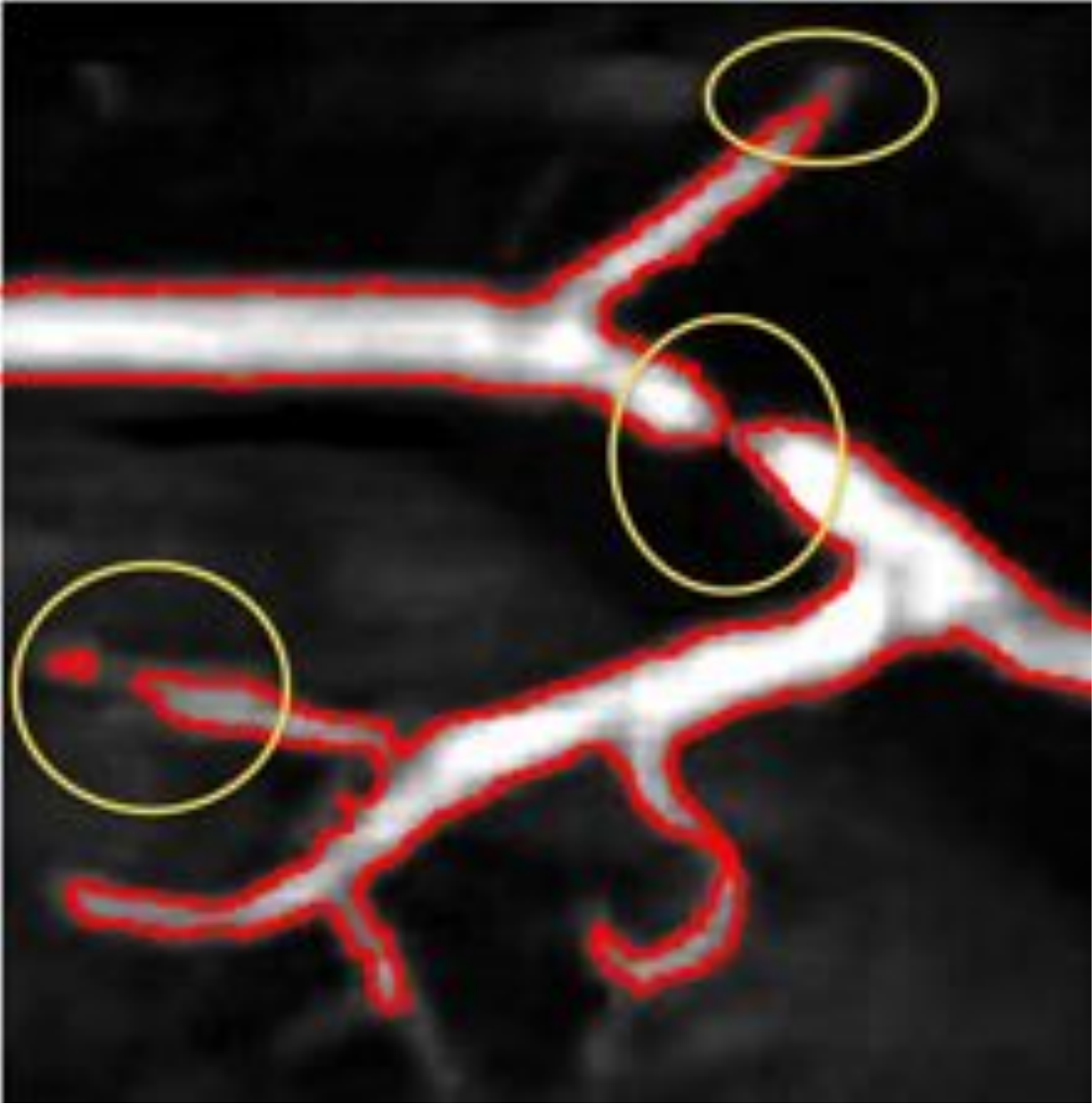}&
\includegraphics[width=41mm, height=41mm]{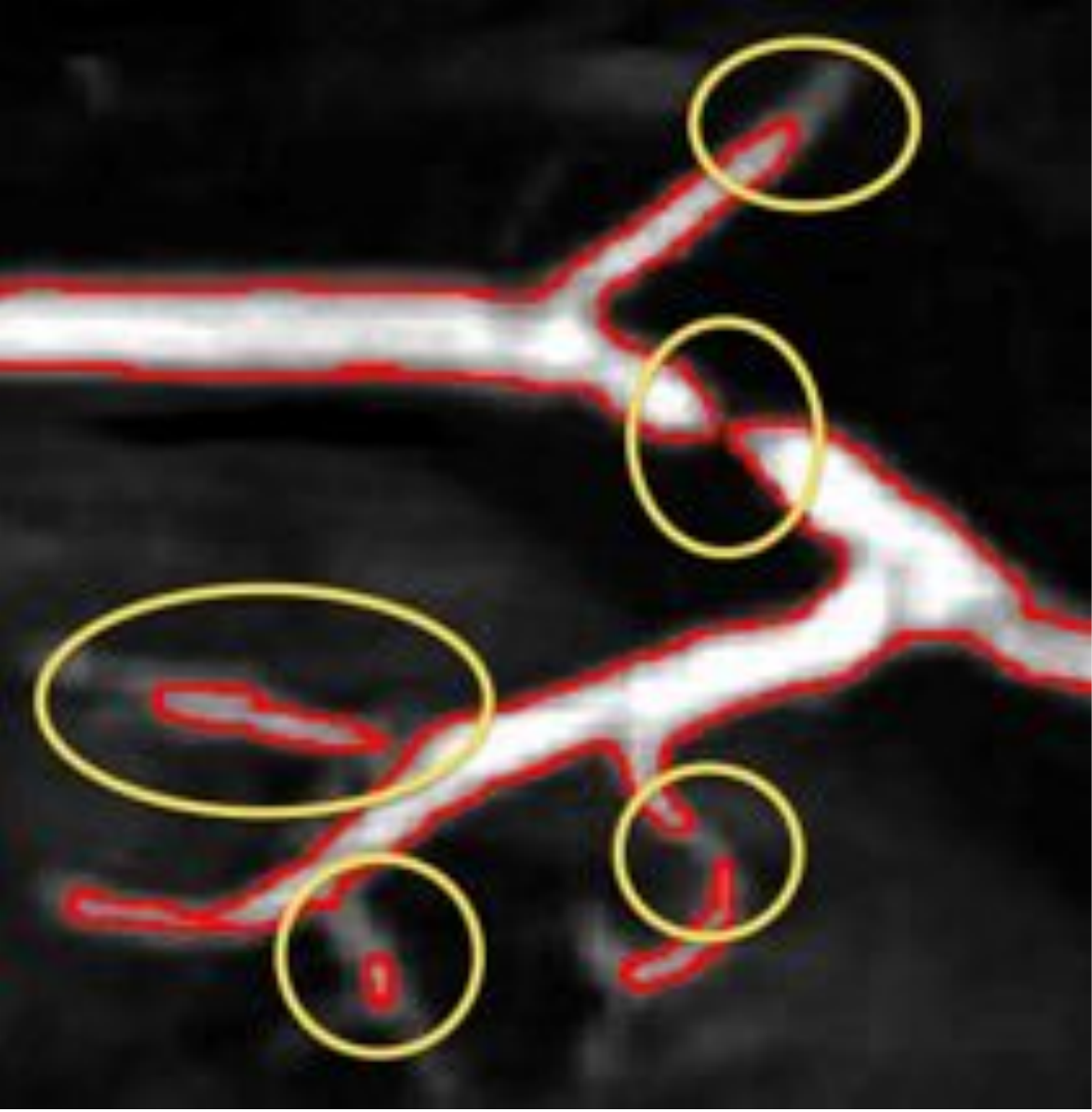}
\\
(a) & (b) & (c)
\end{tabular}
\begin{tabular}{cc}
\includegraphics[width=41mm, height=40mm]{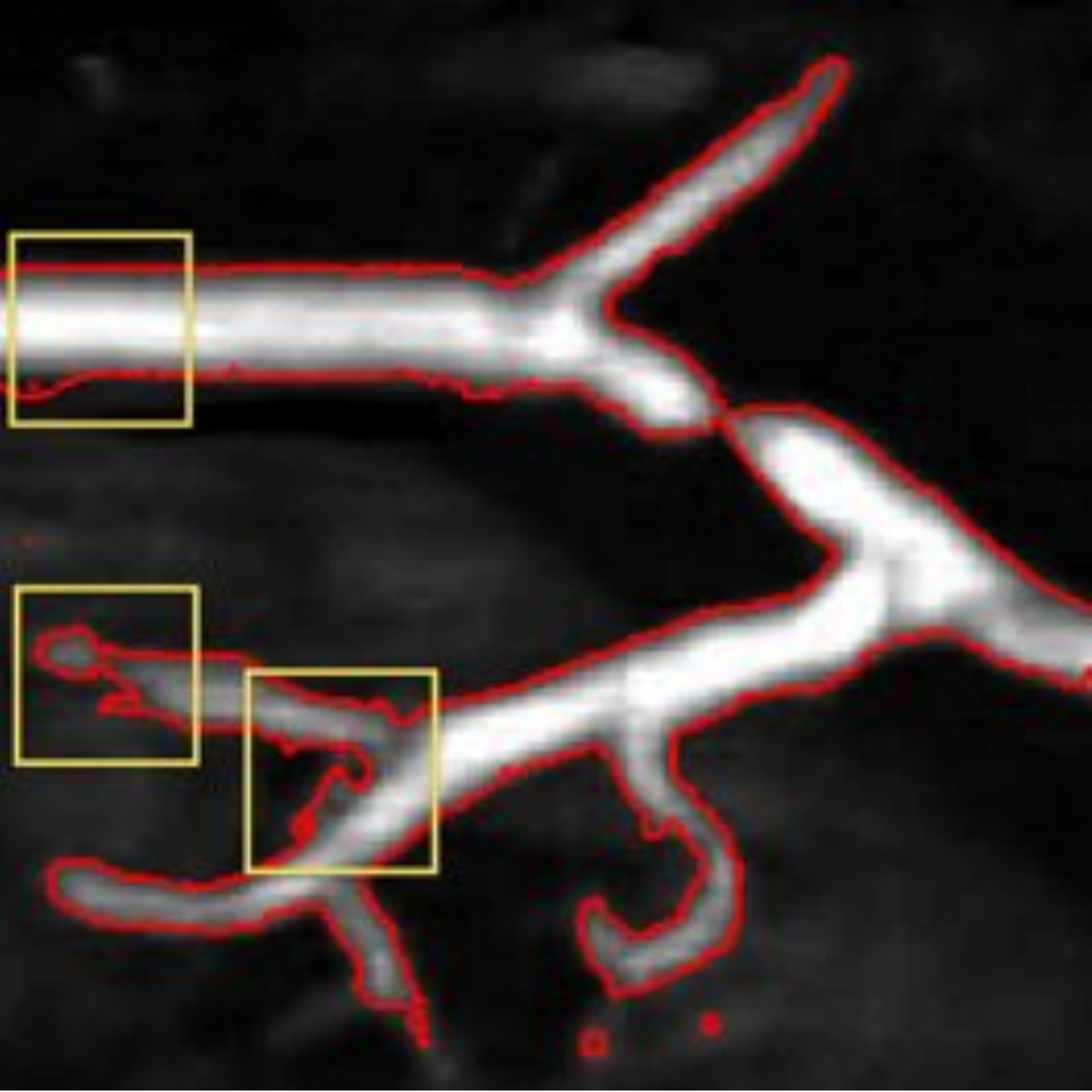} &
\includegraphics[width=41mm, height=40mm]{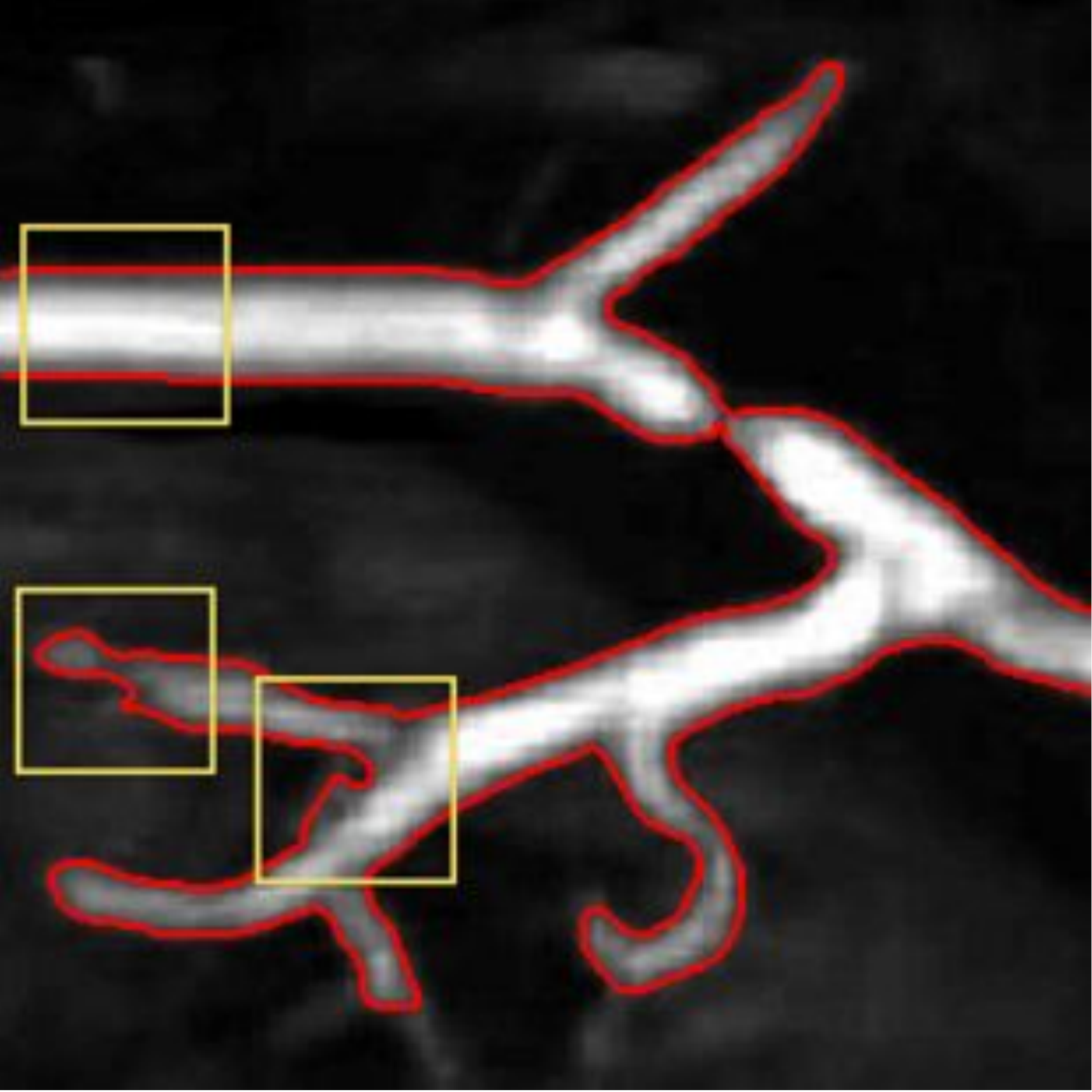}\\
(d) & (e)
\end{tabular}
\begin{tabular}{ccc}
\includegraphics[width=30mm, height=30mm]{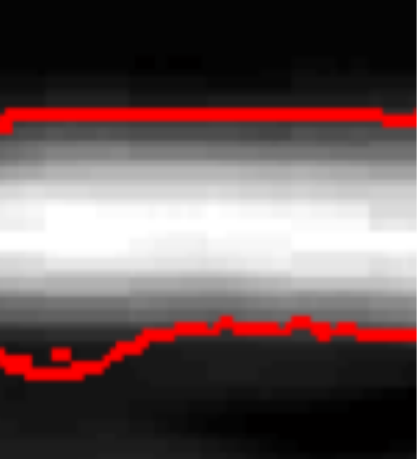}&
\includegraphics[width=30mm, height=30mm]{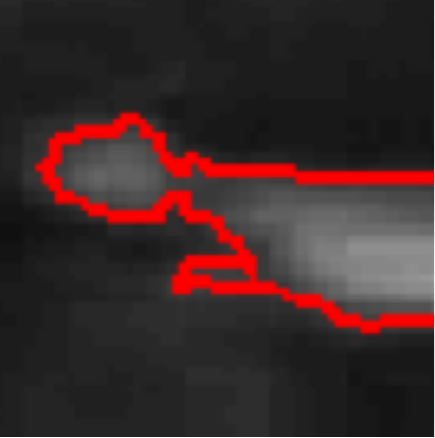} &
\includegraphics[width=30mm, height=30mm]{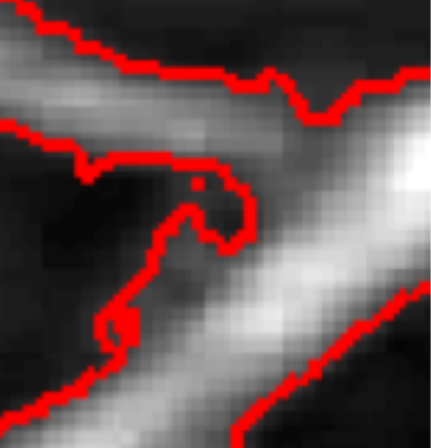} \\
(f) & (g) & (h) \\
\includegraphics[width=30mm, height=30mm]{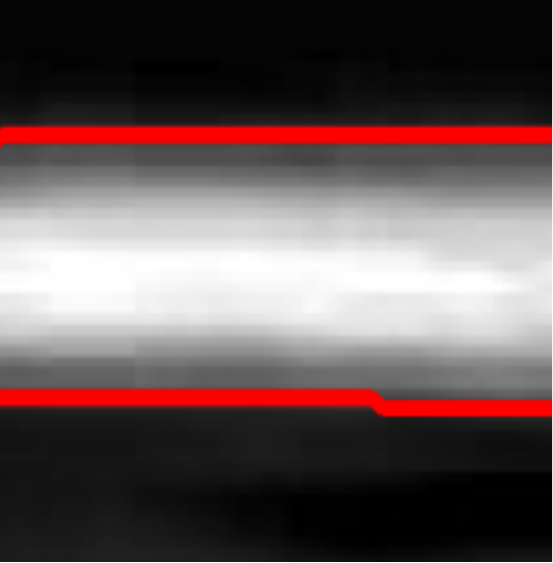}&
\includegraphics[width=30mm, height=30mm]{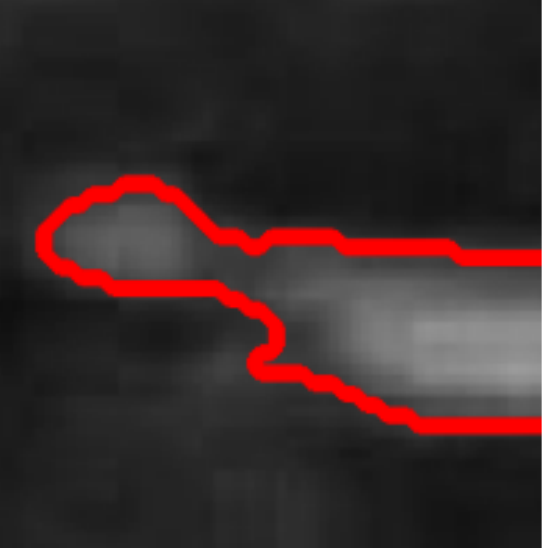} &
\includegraphics[width=30mm, height=30mm]{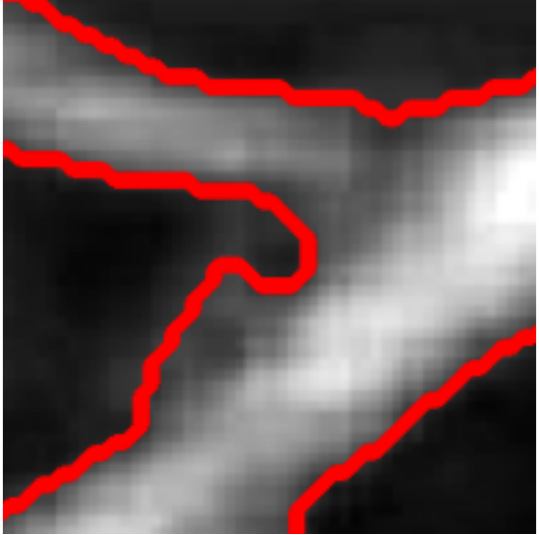}\\
(i) & (j) & (k)
\end{tabular}
\end{center}
\caption{Kidney vascular system segmentation. (a) Given image.
(b), (c) and (d) Results by the methods in \cite{CV}, \cite{DCS} and
\cite{ESF01} respectively. (e) Results by our method.
(f)--(k) are the zoomed-in parts of (d) and (e).} \label{fig_kidney}
\end{figure}

\bigskip
{\it Example 3.} This is a 3D example where we extracted
a volumetric data set of size $201\times 201\times 201$
from a $436\times 436 \times 540$ CTA
(Computed Tomographic Angiography) image of the kidney
vasculature system, see Fig. {\ref{fig_kidney3d}}(a).
Because of different curvatures, diameters, bifurcations, and weak
surfaces impaired by noise, it is hard to detect
the tips of the thin vessels.
Figs. {\ref{fig_kidney3d}}(b) and (c) give the results by
using the method in \cite{ESF02} and our proposed method
respectively.  The figure shows that our method
can give much more details. A visual comparison of the given
image with our result shows that almost all the vessels
are correctly segmented. We note that our method converges in 9 iterations,
see the third column in Table \ref{cardi}.

\begin{figure}[!htb]
\begin{center}
\begin{tabular}{c}
\includegraphics[width=120mm, height=65mm]{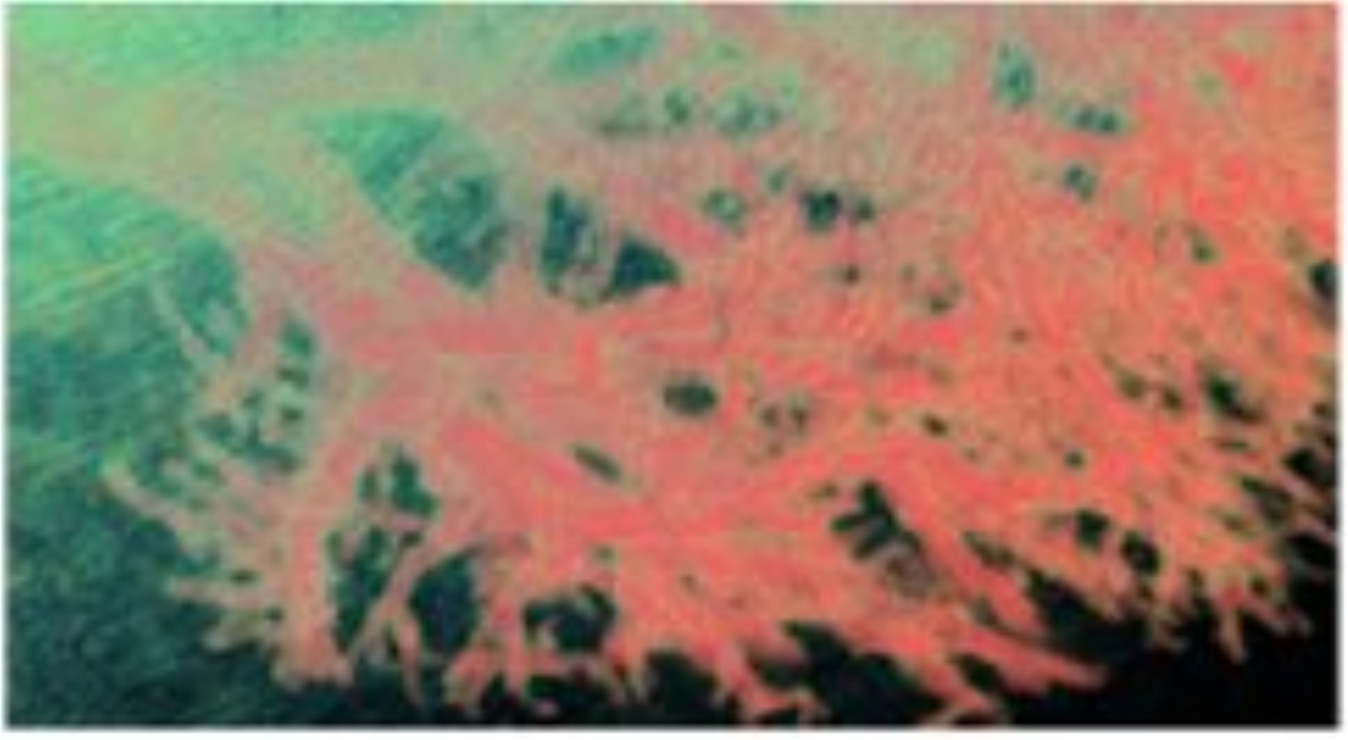}\\
(a)
\end{tabular}
\begin{tabular}{cc}
\includegraphics[width=80mm, height=80mm]{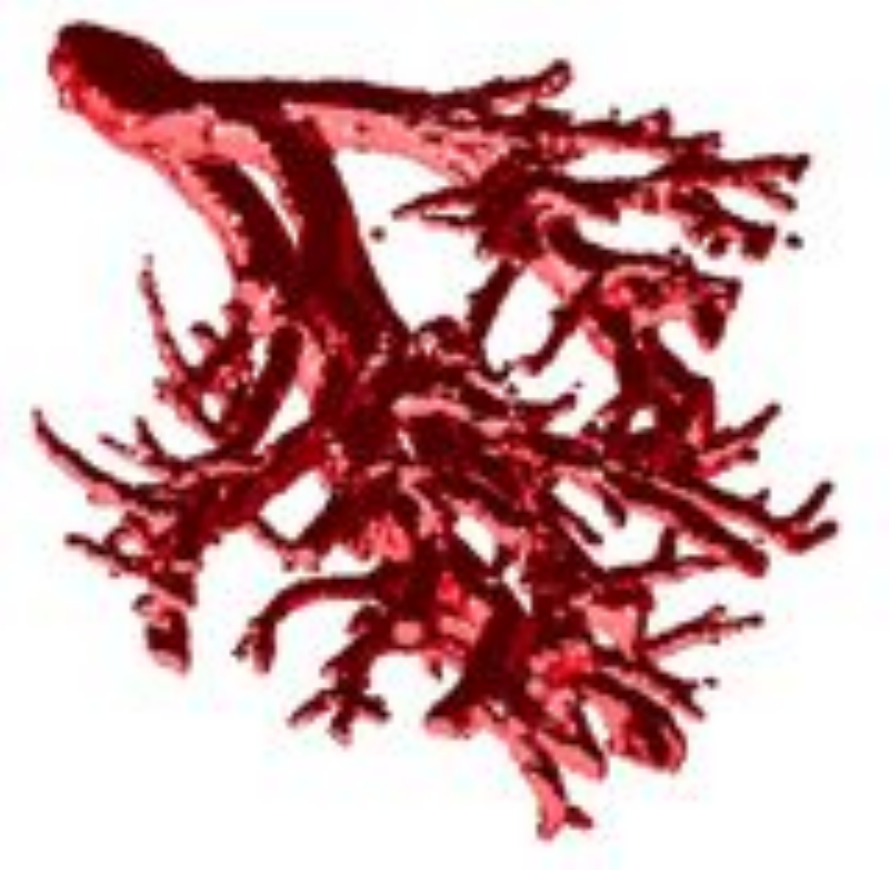} &
\includegraphics[width=80mm, height=80mm]{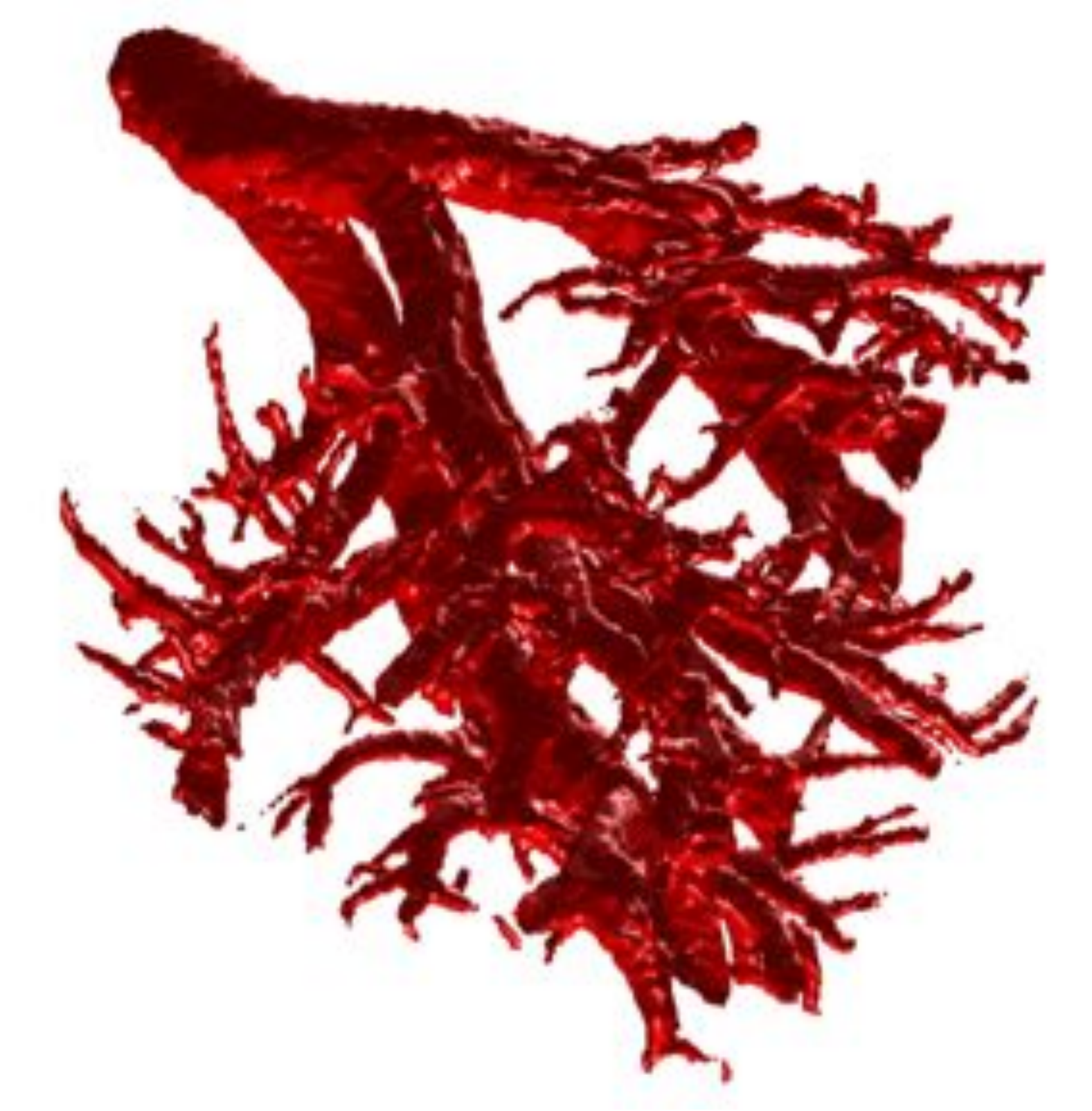} \\
(b) & (c)
\end{tabular}
\end{center}
\caption{Segmentation of the kidney volume data set.
(a) Given CTA image. (b) Result by the method in \cite{ESF02}.
(c) Result by our method.
} \label{fig_kidney3d}
\end{figure}

\bigskip
{\it Example 4.} Our second 3D example is
an MRA data set of a brain aneurysms (vessel wall dilatations).
The $120\times 250 \times 200$ volumetric data set
has been extracted from a $120\times 448\times 540$ MRA
image of the brain-neck vasculature system, see
Fig. {\ref{fig_brain3d}}(a). As in Example 3, the different
curvatures, diameters and bifurcations of the vessels
make it a difficult problem. In addition, the high noise
makes the thin vessels hard to see even by the naked eyes.
For this topological complex tubular structures,
Figs. {\ref{fig_brain3d}}(b) and (c) give the results by the method
in \cite{ESF01} and our method respectively. Obviously, our method
can segment many more thin vessels, especially those corrupted
by the high noise. The results reflect  that our new method is very
effective in handling the noise spread on the surface of the vessels.
For this complicated example, our method converges in 9
iterations only, see the fourth column of Table \ref{cardi}.
Our segmentation technique can be used to compute vessel radii
and other clinically useful measurements in case of aneurysms.
In Fig. {\ref{fig_brain3d}}(c), one may argue that there are some small isolated
points in the image and the surface of the vessels is not
smooth. This can easily be remedied by smoothing our final binary image
by the tight-frame formula (\ref{q2}) one time before we show the image,
since (\ref{q2}) has the denoising property
as we explained in Section \ref{sec2}.
See Fig. {\ref{fig_brain3d}}(d)
for the denoised-and-smoothed image after one iteration of (\ref{q2}).

\begin{figure}[!htb]
\begin{center}
\begin{tabular}{c}
\includegraphics[width=100mm, height=55mm]{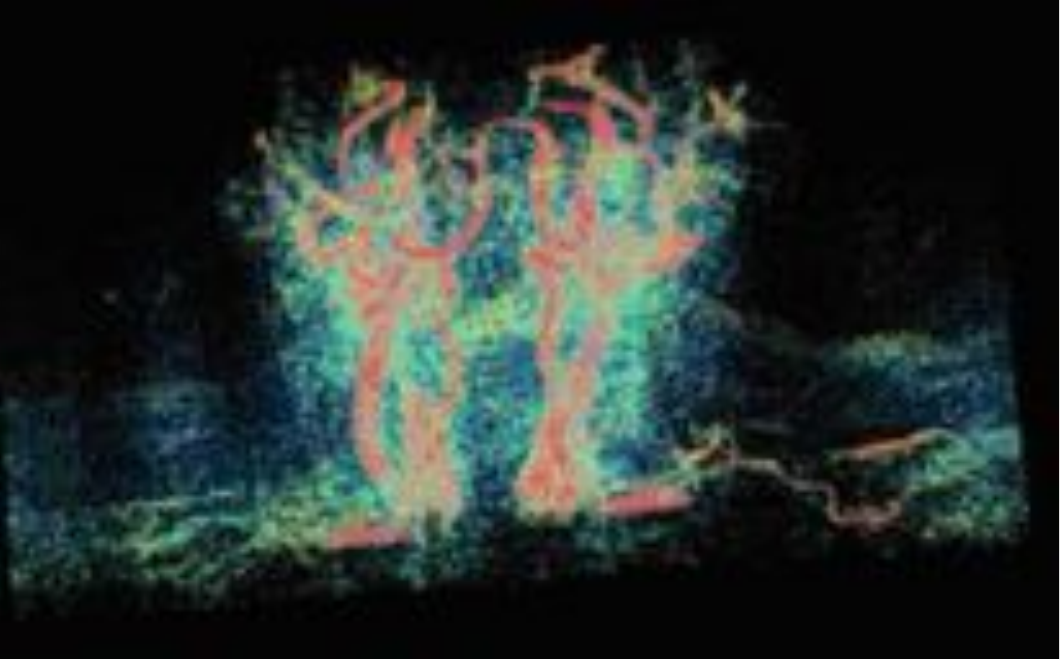} \\
(a)
\end{tabular}
\begin{tabular}{cc}
\includegraphics[width=70mm, height=70mm]{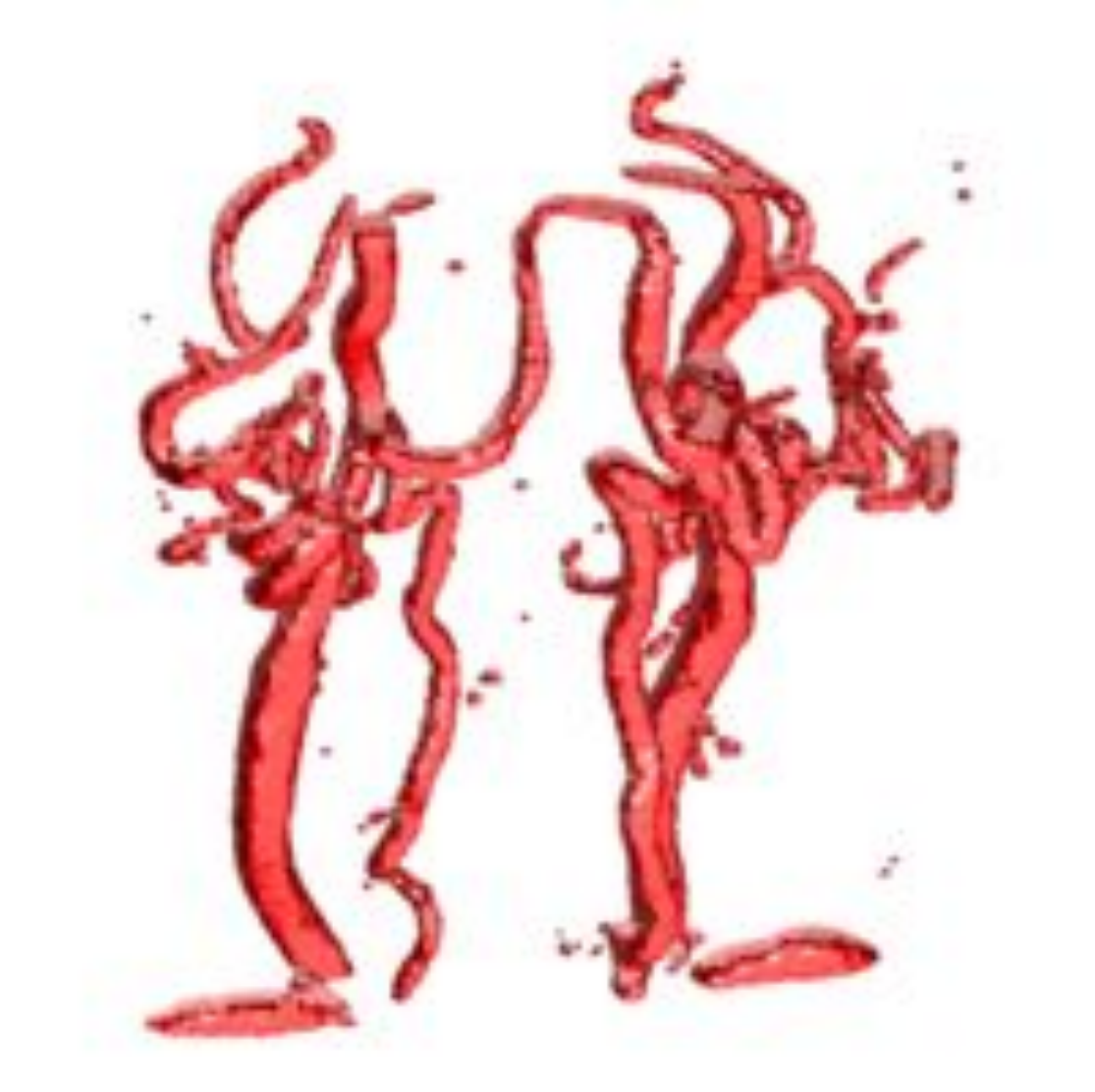} &
\includegraphics[width=74mm, height=70mm]{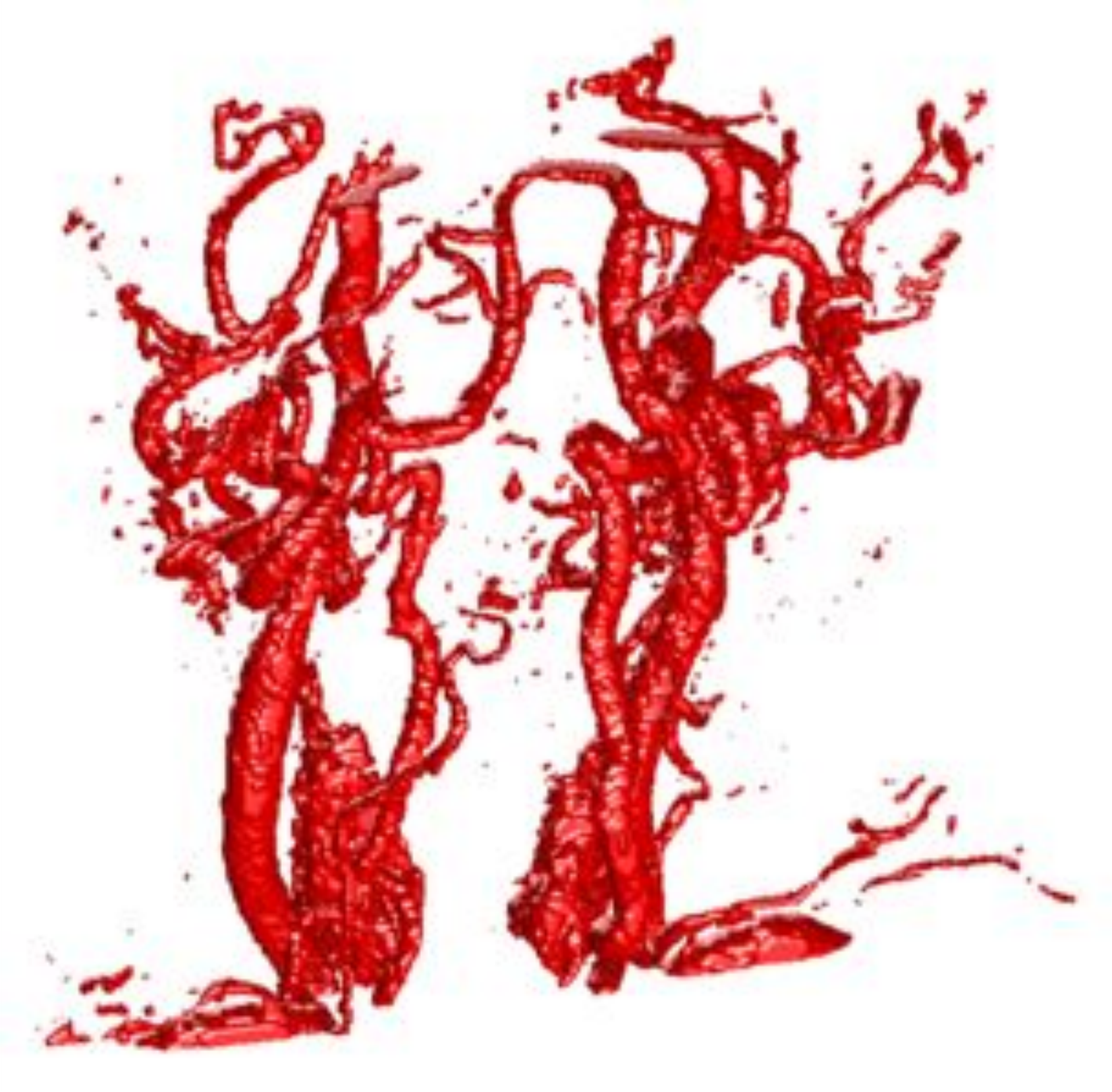} \\
(b) & (c)
\end{tabular}
\begin{tabular}{c}
\includegraphics[width=74mm, height=70mm]{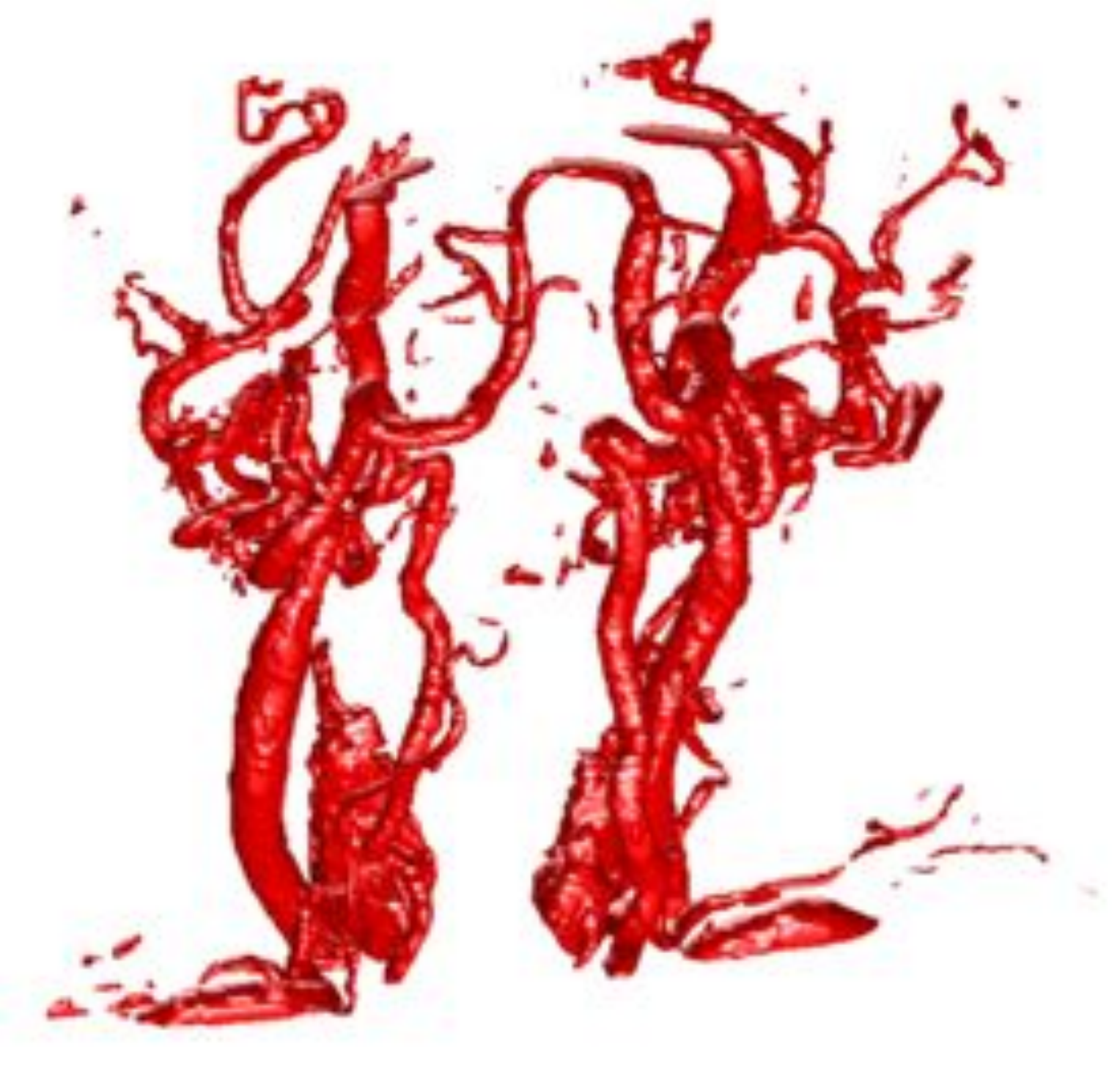} \\
(d)
\end{tabular}
\end{center}
\caption{Segmentation of the brain volume data set.
(a) Given MRA image. (b) Result by the method in \cite{ESF01}.
(c) Result by our method. (d) Result of our method after smoothing
by (\ref{q2}) once.
} \label{fig_brain3d}
\end{figure}

\section{Conclusions and Future Work}\label{sec6}

In this paper, we introduced a new and
efficient segmentation method based on the tight-frame approach.
The numerical results demonstrate
the ability of our method in segmenting tubular structures.
The method can be implemented fast and give very accurate,
smooth boundaries or surfaces. In addition, since the pixel values
of more and more pixels will be set to either 0 or 1 during the iterations,
by taking advantage of this, one can construct a sparse
data structure to accelerate the method. Moreover, one
can use different tight-frame systems such as those from
contourlets, curvelets or steerable-wavelet \cite{DV,CDDL,FE}
to capture more directions along the boundary.
Though we have proved that our algorithm will always
converge to a binary image, it will be interesting to
see what functional the binary image is minimizing.
The framework for proving convergence for tight-frame
algorithms, as developed in \cite{JRZ}, may be useful here.
These are the directions we will explore in the future.

\subsubsection*{Acknowledgments.}
This work was supported by HKRGC Grant CUHK 400510 and CUHK Direct
Allocation Grant 2060408.


\end{document}